\documentclass[12pt,a4paper]{article}

\usepackage{graphicx,tikz}
\usepackage{latexsym}
\usepackage{amsmath}
\usepackage{amssymb}
\usepackage{arydshln}

\usepackage{cite}
\usepackage{stmaryrd}
\usepackage{enumerate}

\usepackage{hyperref}

\usepackage{color}
\usepackage{lineno}
\usepackage{graphicx}
\usepackage{ae}
\usepackage{amsmath}
\usepackage{amssymb}
\usepackage{latexsym}
\usepackage{url}
\usepackage{epsfig}
\usepackage{mathrsfs}
\usepackage{amsfonts}
\usepackage{amsthm}

\newcommand{\Cay}{\mathrm{Cay}}

\newtheorem{theorem}{Theorem}[section]

\newtheorem{lemma}[theorem]{Lemma}

\newtheorem{cor}[theorem]{Corollary}

\newtheorem{example}{Example}

\theoremstyle{definition}

\numberwithin{equation}{section} 
\allowdisplaybreaks    

\def\qed{\hfill$\Box$\vspace{12pt}}

\long\def\delete#1{}

\usepackage{xcolor}
\usepackage[normalem]{ulem}

\begin{document}
\title{Fractional revival on Cayley graphs over abelian groups}
\author{Jing Wang$^{a,b,c}$,~Ligong Wang$^{a,b,c}$$^,$\thanks{Supported by the National Natural Science Foundation of China (Nos. 11871398 and 12271439).}~,~Xiaogang Liu$^{a,b,c,}$\thanks{Supported by the National Natural Science Foundation of China (No. 12371358) and the Guangdong Basic and Applied Basic Research Foundation (No. 2023A1515010986).}~$^,$\thanks{Corresponding author. Email addresses: wj66@mail.nwpu.edu.cn, lgwangmath@163.com, xiaogliu@nwpu.edu.cn}
\\[2mm]
{\small $^a$School of Mathematics and Statistics,}\\[-0.8ex]
{\small Northwestern Polytechnical University, Xi'an, Shaanxi 710072, P.R.~China}\\
{\small $^b$Research \& Development Institute of Northwestern Polytechnical University in Shenzhen,}\\[-0.8ex]
{\small Shenzhen, Guangdong 518063, P.R. China}\\
{\small $^c$Xi'an-Budapest Joint Research Center for Combinatorics,}\\[-0.8ex]
{\small Northwestern Polytechnical University, Xi'an, Shaanxi 710129, P.R. China}
}
\date{}

\openup 0.5\jot
\maketitle

\begin{abstract}
In this paper, we investigate the existence of fractional revival on Cayley graphs over finite abelian groups. We give a necessary and sufficient condition for Cayley graphs over finite abelian groups to have fractional revival. As applications, the existence of fractional revival on circulant graphs and cubelike graphs are characterized.
\smallskip


\end{abstract}

\section{Introduction}
  Quantum state transfer in quantum networks first introduced by Bose in \cite{Bose03} is a very important research content for quantum communication protocols. Let $G=(V(G), E(G))$ be a graph with the adjacency matrix $A_G$. The \emph{transition matrix} \cite{FarhiG98} of $G$ with respect to $A_G$ is defined by
$$
H_{A_{G}}(t) = \exp(\mathrm{i}tA_{G})=\sum_{k=0}^{\infty}\frac{\mathrm{i}^{k} A^{k}_{G} t^{k}}{k!}, ~ t \in \mathbb{R},~\mathrm{i}=\sqrt{-1}.
$$
Let $\mathbf{e}_u$ denote the standard basis vector in $\mathbb{C}^{|G|}$ indexed by the vertex $u$, where $|G|$ denotes the number of vertices of $G$. If $u$ and $v$ are distinct vertices in $G$ and there is a time $t$ such that
\begin{equation}\label{FRCAY-EQUATION1}
H_{A_{G}}(t)\mathbf{e}_u=\gamma\mathbf{e}_v,
\end{equation}
where $\gamma$ is a complex number and $|\gamma|=1$, then we say that $G$ has \emph{perfect state transfer} (PST for short) from $u$ to $v$ at time $t$. In particular, if $u=v$ in Equation (\ref{FRCAY-EQUATION1}), then we say $G$ is \emph{periodic} at vertex $u$ at time $t$. Let $H_{A_G}(t)_{u,v}$ denote the $(u,v)$-entry of $H_{A_G}(t)$, where $u,v\in V(G)$. It is easy to see that Equation (\ref{FRCAY-EQUATION1}) is equivalent to
\begin{equation*}\label{FRCAY-EQUATION2}
|H_{A_G}(t)_{u,v}|=1.
\end{equation*}

To characterize which graphs having PST has attracted the attention of more and more researchers. Up until now, many wonderful results have been achieved on  diverse families of graphs, including trees \cite{Bose03, CoutinhoL2015, Fan, GodsilKSS12}, Cayley graphs \cite{Basic11, Basic09, CaoCL20, CaoF21, CaoWF20, CC, LiLZZ21, HPal3, HPal, HPal5, Tan19, Tm19}, distance regular graphs \cite{Coutinho15} and some graph operations such as NEPS \cite{chris1, chris2,  LiLZZ21, HPal1, HPal4, SZ}, coronas \cite{AckBCMT16}, joins \cite{Angeles10} and total graphs \cite{LiuW2021}. For more information, we refer the reader to \cite{Coh14, Coh19, CGodsil, Godsil12, LiuZ2022, HZ, Zhou14}. However, the graphs having PST are still quite rare. In \cite[Corollary~6.2]{Godsil12}, Godsil has shown that there are at most finitely many connected graphs with a given maximum valency where PST occurs.

Recently, another physical phenomenon on quantum state transfers, the \emph{fractional revival} (FR for short), was introduced in \cite{RS15}, which is a generalization of PST. We say a graph $G$ has \emph{fractional revival} at time $t$ from $u$ to $v$, if
\begin{equation}\label{FRCAY-EQUATION3}
H_{A_{G}}(t)\mathbf{e}_u=\alpha\mathbf{e}_u+\beta\mathbf{e}_v,
\end{equation}
where $\alpha, \beta$ are complex numbers and $|\alpha|^2+|\beta|^2=1$. In particular, if $\alpha=0$, then $G$ has PST from $u$ to $v$ at time $t$; and if $\beta=0$, then $G$ is periodic at vertex $u$ at time $t$. It is easy to verify that Equation (\ref{FRCAY-EQUATION3}) is equivalent to
 \begin{equation*}\label{FRCAY-EQUATION434}
 |H_{A_G}(t)_{u,u}|^2+|H_{A_G}(t)_{u,v}|^2=1.
 \end{equation*}
It was shown in \cite{CCTVZ19} that if a graph $G$ has FR from $u$ to $v$, then $G$ has FR from $v$ to $u$ at the same time. Thus, we simply say $G$ has FR between $u$ and $v$ at time $t$.
Up until now, only few results on FR has been given, which are listed as follows:
\begin{itemize}
  \item A systematic study of FR at two sites in XX quantum spin chains was given in \cite{GenestVZ16}.
  \item An example of a graph that has balanced FR between antipodes was given in \cite{BernardCLTV18}.
  \item A cycle has FR if and only if it has four or six vertices, and a path has FR if and only it has two, three or four vertices \cite{CCTVZ19}.
  \item FR in graphs, whose adjacency matrices belong to the Bose--Mesner algebra of association schemes, was given in \cite{ChanCCTVZ20}.
  \item An indication on how FR can be swapped with PST by modifying chains through isospectral deformations was given in \cite{SVZ22}.
 \end{itemize}

Let $\Gamma$ be a finite abelian group. It is well-known that $\Gamma$ can be decomposed as a direct sum of cyclic groups:
\begin{equation}\label{GammaDef-1}
\Gamma=\mathbb{Z}_{n_1}\oplus \mathbb{Z}_{n_2}\oplus\cdots\oplus \mathbb{Z}_{n_r}~(n_s\geq2),
\end{equation}
where $\mathbb{Z}_m=(\mathbb{Z}/m\mathbb{Z},+)=\{0,1,2,\ldots,m-1\}$ is the additive group of integers modulo $m$. Let $S$ be a subset of $\Gamma$ such that $0\notin S$ and $-S=S$. The \emph{Cayley graph} over $\Gamma$ with the connection set $S$, denoted by $\Cay(\Gamma,S)$, is defined to be the graph whose vertices are the elements of $\Gamma$ and two vertices $u,v\in \Gamma$ are adjacent if and only if $u-v\in S$.
Note that $0\notin S$ means that $\Cay(\Gamma,S)$ has no loop and $-S=S$ means that $\Cay(\Gamma,S)$ is an undirected graph.

In this paper, our purpose is to find the existence of FR on Cayley graph $\Cay(\Gamma,S)$. With a slight abuse of notation, for $x=(x_1, x_2,\ldots, x_r)\in \Gamma$ and $y=(y_1,y_2,\ldots, y_r)\in \Gamma$, let $x\dot>y$ denote that there exists an $i\in\{1,2,\ldots,r\}$ such that $x_1=y_1,\ldots, x_{i-1}=y_{i-1}$ and $x_i>y_i$ (Here, we regard $x_i$ and $y_i$ as integers). Hence any two distinct elements $x$ and $y$ are comparable with this operation. Let
$$\mathrm{wt}(x)=\sum_{s=1}^r x_s, ~\mathrm{wt}(xy)=\sum_{s=1}^r x_sy_s \text{~~and~~} \mathrm{wt}\left(\frac{x}{y}\right)=\sum\limits_{s=1}^{r}\frac{x_s}{y_s}$$
denote the sum of all the elements of $x=(x_1, x_2,\ldots, x_r)$,  $xy=(x_1y_1, x_2y_2,\ldots, x_ry_r)$ and $\frac{x}{y}=(\frac{x_1}{y_1},\frac{x_2}{y_2},\ldots, \frac{x_r}{y_r})$, respectively.

Our result is stated as follows, which gives a necessary and sufficient condition for $\Cay(\Gamma,S)$ to have FR.

\begin{theorem}\label{FRCAY-mian result}
Let $\Gamma=\mathbb{Z}_{n_1}\oplus \mathbb{Z}_{n_2}\oplus\cdots \oplus \mathbb{Z}_{n_r}$ be a finite abelian group of order $|\Gamma|=n_1n_2\cdots n_r$ as shown in (\ref{GammaDef-1}), and $S$ a subset of $\Gamma$ such that $0\notin S$ and $-S=S$. Let $n=(n_1,n_2,\ldots, n_r)$ and
$$
N=\left\{(x,y)\left|~x,y\in \Gamma,~ x\dot>y,~ \mathrm{wt}\left(\frac{2a(x-y)}{n}\right)\text{~is~even}\right.\right\}.
$$
Then $\Cay(\Gamma, S)$ has FR at time $t$ between vertices $v$ and $v+a$ if and only if the following three conditions hold:
\begin{itemize}
\item[\rm(a)] $a=(a_1, a_2,\ldots, a_r)$ is of order two;
\item[\rm(b)] $n_s$ is even for $a_s\neq 0$, $s=1,2,\ldots, r$;
\item[\rm(c)] $\frac{t}{2\pi}(\lambda_x-\lambda_y)\in \mathbb{Z}, \text{~for~all~} (x,y)\in N$, where $\lambda_x=\sum\limits_{g\in S}\prod\limits_{s=1}^re^{\frac{2\pi \mathrm{i}}{n_s}x_sg_s}$.
\end{itemize}
Moreover, if these conditions hold, and $\lambda_x-\lambda_y$ are integers for all $(x,y)\in N$, then there is a minimum time $t=\frac{2\pi}{M}$ at which FR occurs between $u$ and $v$, where
$$M=\gcd(\{\lambda_x-\lambda_y\}_{(x,y)\in N}).$$
\end{theorem}


\section{Proof of Theorem \ref{FRCAY-mian result}}

In this section, we give the proof of Theorem \ref{FRCAY-mian result}. Before proceeding, we first give some definitions and notions.

Let $\mathbb{Z}$ and $\mathbb{C}$ denote the set of integer numbers and complex numbers, respectively. Let $G$ be a graph with adjacency matrix $A_G$.  The eigenvalues of $A_G$ are called the \emph{eigenvalues} of $G$. Suppose that $ \lambda_1\geq\lambda_2\geq\cdots\geq\lambda_{|G|}$ ($\lambda_i$ and $\lambda_j$ may be equal) are all eigenvalues of $G$  and $\mathbf{x}_j$ is the eigenvector associated with $\lambda_{j}$, $j=1,2,\ldots,|G|$.  Let $\mathbf{x}^H$ denote the conjugate transpose of a column vector $\mathbf{x}$. Then, for each eigenvalue $\lambda_j$ of $G$, define
$$
E_{\lambda_j} =\mathbf{x}_j (\mathbf{x}_j)^H,
$$
which is usually called the \emph{eigenprojector} corresponding to  $\lambda_j$ of $G$. Note that $\sum_{j=1}^{|G|}E_{\lambda_j}=I$ (the identity matrix). Then
\begin{equation}
\label{spect1}
A_G=A_G\sum_{j=1}^{|G|}E_{\lambda_j} =\sum_{j=1}^{|G|}A_G\mathbf{x}_j (\mathbf{x}_j)^H  =\sum_{j=1}^{|G|}\lambda_j\mathbf{x}_j (\mathbf{x}_j)^H  =\sum_{j=1}^{|G|}\lambda_jE_{\lambda_j},
\end{equation}
which is called the \emph{spectral decomposition of $A_G$ with respect to the eigenvalues}  (see ``Spectral Theorem for Diagonalizable Matrices'' in \cite[Page 517]{MAALA}). Note that $E_{\lambda_j}^{2}=E_{\lambda_j}$ and $E_{\lambda_j}E_{\lambda_h}=\mathbf{0}$ for $j\neq h$, where $\mathbf{0}$ denotes the zero matrix. So, by (\ref{spect1}), we have
\begin{equation*}\label{SpecDec2-1}
H_{A_G}(t)=\sum_{k\geq 0}\dfrac{\mathrm{i}^{k}A_G^{k}t^{k}}{k!}=\sum_{k\geq 0}\dfrac{\mathrm{i}^{k}\left(\sum\limits_{j=1}^{|G|}\lambda_{j}^{k}E_{\lambda_j}\right)t^{k}}{k!} =\sum_{j=1}^{|G|}\exp(\mathrm{i}t\lambda_{j})E_{\lambda_j}.
\end{equation*}

 Let $\Gamma=\mathbb{Z}_{n_1}\oplus \mathbb{Z}_{n_2}\oplus\cdots\oplus \mathbb{Z}_{n_r}~(n_s\geq2)$ be a finite abelian group. For every $x=(x_1,\ldots, x_r)\in \Gamma$, $(x_s\in \mathbb{Z}_{n_s})$, the mapping
 \begin{equation}\label{character}
 \chi_x: \Gamma\rightarrow \mathbb{C}, ~\chi_x(g)=\prod_{s=1}^r\omega_{n_s}^{x_sg_s}, ~(\text{for~} g=(g_1,\ldots,g_r)\in \Gamma)
 \end{equation}
is a \emph{character} of $\Gamma$, where $\omega_{n_s}=\exp(2\pi \mathrm{i}/n_s)$ is a primitive $n_s$-th root of unity in $\mathbb{C}$. It is easy to verify that
$$
\chi_x(g)=\chi_g(x) \text{~~for~all~}x,g\in \Gamma.
$$
Let $\hat{\Gamma}=\{\chi_x | x\in \Gamma\}$ be the \emph{dual group} or \emph{character group} of $\Gamma$. Clearly, the mapping
$$
\Gamma\rightarrow\hat{\Gamma},~x\mapsto\chi_x
$$
is a group isomorphism.

\medskip

Now we give the proof of Theorem \ref{FRCAY-mian result}.

\medskip

\begin{Tproof}\textbf{~of~Theorem~\ref{FRCAY-mian result}.}
We first prove (a) and (b) are necessary for $G:=\Cay(\Gamma, S)$ to have FR. For a vertex $u\in V(G)$, let $\mathrm{Aut}(G)_u$ denote  the group of automorphisms of $G$ that fix $u$. It is shown in \cite[Prposition 6.4]{CCTVZ19} that if $G$ has FR between distinct vertices $u$ and $v$, then $\mathrm{Aut}(G)_u=\mathrm{Aut}(G)_v$, that is, any automorphism of $G$ that fixes $u$ must fix $v$. Note that the mapping that sends $x$ to $-x$ is an automorphism of $\Gamma$, which is also an automorphism of $G$.  The fixed points of this automorphism are the elements of $G$ with order one
or two. So if $G$ has FR between vertices $0$ and $a=(a_1,a_2,\ldots, a_r)$, then $a$ has order two. Recall that $n=(n_1,n_2,\ldots,n_r)$, then
$$2(a_1,a_2,\ldots, a_r)\equiv 0\pmod{(n_1,n_2,\ldots,n_r)}.$$
Therefore, we have $a_s=n_s/2$ or $a_s=0$. Together with $a\neq(0,0,\ldots, 0)$, this implies that $n_s$ is even for $a_s\neq 0$, $s=1,2,\ldots,r$.

Next we prove that if (a) and (b) hold, then $G$ has FR at time $t$ between vertices $v$ and $v+a$ if and only if (c) holds.

Recall that the order of $G$ is $|\Gamma|$. By \cite[Theorem 5.4.10]{Steinberg12}, take a matrix $P=\frac{1}{\sqrt{|\Gamma|}}(\chi_g(h))_{g,h\in G}$. Let $p_g$ be the $g$-th column of $P$. Then the adjacency matrix $A_G$ and the transition matrix $H_{A_G}(t)$ can be written as follows:
$$
A_G=\sum_{g\in G}\lambda_g E_g,
$$
and
$$
H_{A_G}(t)=\sum_{g\in G}\exp(\mathrm{i}\lambda_g t)E_g,
$$
where
$$
\lambda_g=\sum_{s\in S}\chi_g(s), ~g\in G,
$$
and
$$
E_g=p_gp_g^H=\frac{1}{|\Gamma|}(\chi_g(u-v))_{u,v\in G}.
$$
Then, for two vertices $u, v\in G$, we have
$$H_{A_G}(t)_{u,v}=\sum_{g\in G}\exp(\mathrm{i}\lambda_g t)(E_g)_{u,v}=\frac{1}{|\Gamma|}\sum_{g\in G}\exp(\mathrm{i}\lambda_gt)\chi_g(u-v).$$
In particular, if $u=v$, we have $\chi_g(u-u)=\chi_g(0)=1$, and
$$H_{A_G}(t)_{u,u}=\sum_{g\in G}\exp(\mathrm{i}\lambda_g t)(E_g)_{u,u}=\frac{1}{|\Gamma|}\sum_{g\in G}\exp(\mathrm{i}\lambda_gt).$$
Thus, for $u-v=a=(a_1,a_2,\ldots,a_r)$, we have
\begin{align}\label{FR-htuu+htuv} \nonumber
& |H_{A_G}(t)_{uu}|^2 + |H_{A_G}(t)_{uv}|^2 \\ \nonumber
&=\left|\frac{1}{|\Gamma|}\sum_{g\in G}\exp(\mathrm{i}\lambda_gt)\right|^2+\left|\frac{1}{|\Gamma|}\sum_{g\in G}\exp(\mathrm{i}\lambda_gt)\chi_g(a)\right|^2 \\ \nonumber
&=\frac{1}{|\Gamma|^2}\left(\left|\sum_{g\in G}\exp(\mathrm{i}\lambda_gt)\right|^2+\left|\sum_{g\in G}\exp(\mathrm{i}\lambda_gt)\chi_g(a)\right|^2 \right) \\ \nonumber
&=\frac{1}{|\Gamma|^2}\left(\left(\sum_{g\in G}\exp(\mathrm{i}\lambda_gt)\right) \left(\overline{\sum_{g\in G}\exp(\mathrm{i}\lambda_gt)}\right) + \left(\sum_{g\in G}\exp(\mathrm{i}\lambda_gt)\chi_g(a)\right) \left(\overline{\sum_{g\in G}\exp(\mathrm{i}\lambda_gt)\chi_g(a)}\right) \right) \\ \nonumber
&=\frac{1}{|\Gamma|^2}\left(\left(\sum_{g\in G}\exp(\mathrm{i}\lambda_gt)\right) \left(\sum_{g\in G}\overline{\exp(\mathrm{i}\lambda_gt)}\right) + \left(\sum_{g\in G}\exp(\mathrm{i}\lambda_gt)\chi_g(a)\right) \left(\sum_{g\in G}\overline{\exp(\mathrm{i}\lambda_gt)}\,\overline{\chi_g(a)} \right)\right) \\ \nonumber
&=\frac{1}{|\Gamma|^2}\left(\sum_{x,y\in G}\exp(\mathrm{i}t(\lambda_x-\lambda_y))+ \sum_{x,y\in G}\exp(\mathrm{i}t(\lambda_x-\lambda_y)) \chi_{a}(x-y)\right)\\ \nonumber
&=\frac{1}{|\Gamma|^2}\left(\sum_{x,y\in G}\exp(\mathrm{i}t(\lambda_x-\lambda_y))(\chi_{a}(x-y)+1)\right)\\
&=\frac{1}{|\Gamma|^2}\left(2|\Gamma|+\sum_{x,y\in G,~x\neq y}\exp(\mathrm{i}t(\lambda_x-\lambda_y))(\chi_{a}(x-y)+1)\right).
\end{align}
Recall that $x\dot>y$ denotes that there exists an $i\in\{1,2,\ldots,r\}$ such that $x_1=y_1,\ldots, x_{i-1}=y_{i-1}$ and $x_i>y_i$. Then we divide the set of ordered pairs $(x,y)$ with $x,y\in G$ and $x\neq y$ into two sets $B,D$ of the same order, where $B=\{ (x,y)| x,y\in G,x\dot>y\}$ and $D=\{ (x,y)| x,y\in G,y\dot>x\}$. Moreover, there is a one-to-one correspondence between $B$ and $D$.
Notice that
$$
\overline{\exp(\mathrm{i}t(\lambda_y-\lambda_x))(\chi_{a}(y-x)+1)} =\exp(\mathrm{i}t(\lambda_x-\lambda_y))(\chi_{a}(x-y)+1).
$$
Then (\ref{FR-htuu+htuv}) amounts to
\begin{equation}\label{FRCAY-EQUATION5}
|H_{A_G}(t)_{uu}|^2 + |H_{A_G}(t)_{uv}|^2=\frac{1}{|\Gamma|^2}\left(2|\Gamma|+2\sum_{(x,y)\in B}\mathrm{Re}\left(\exp(\mathrm{i}t(\lambda_x-\lambda_y))(\chi_{a}(x-y)+1)\right)\right),
\end{equation}
where $\mathrm{Re}(*)$ denotes the real part of the complex number $*$. Note that $a=(a_1,a_2,\ldots,a_r)$ and $x-y=((x-y)_1, (x-y)_2,\ldots, (x-y)_r)$. By (\ref{character}), we have
\begin{equation}\label{FRCAY-EQUATION6}
\chi_{a}(x-y)=\exp\left(2\pi \mathrm{i} \sum_{s=1}^r\frac{a_s(x-y)_s}{n_s}\right).
\end{equation}
Recall that $N=\left\{(x,y)\left| x,y\in G, x\dot>y, \mathrm{wt}\left(\frac{2a(x-y)}{n}\right) \text{~is~even}\right.\right\}$ and $a$ is of order $2$. Plugging (\ref{FRCAY-EQUATION6}) into (\ref{FRCAY-EQUATION5}), we have
\begin{align*}
& |H_{A_G}(t)_{uu}|^2 + |H_{A_G}(t)_{uv}|^2 \\ \nonumber
&=\frac{1}{|\Gamma|^2}\left(2|\Gamma|+2\sum_{(x,y)\in B}\mathrm{Re}\left(\exp(\mathrm{i}t(\lambda_x-\lambda_y))+\exp \left(\mathrm{i}t(\lambda_x-\lambda_y)+2\pi \mathrm{i}\sum_{s=1}^r\frac{a_s(x-y)_s}{n_s}\right)\right)\right)\\ \nonumber
&=\frac{1}{|\Gamma|^2}\left(2|\Gamma|+2\sum_{(x,y)\in B}\left(\cos\left(t(\lambda_x-\lambda_y)\right)+\cos\left(t(\lambda_x-\lambda_y)+2\pi\sum_{s=1}^r\frac{a_s(x-y)_s}{n_s}\right)\right)\right)\\ \nonumber
&=\frac{1}{|\Gamma|^2}\left(2|\Gamma|+2\sum_{(x,y)\in B}\left(\cos\left(t(\lambda_x-\lambda_y)\right)+\cos\left(t(\lambda_x-\lambda_y)+\pi \mathrm{wt}\left(\frac{2a(x-y)}{n}\right)\right)\right)\right)\\ \nonumber
&=\frac{1}{|\Gamma|^2}\left(2|\Gamma|+4\sum_{(x,y)\in N}\cos\left(t(\lambda_x-\lambda_y)\right)\right).
\end{align*}
\noindent Therefore, $G$ has FR if and only if
\begin{equation*}
\frac{1}{|\Gamma|^2}\left(2|\Gamma|+4\sum_{(x,y)\in N}\cos\left(t(\lambda_x-\lambda_y)\right)\right)=1,
\end{equation*}
that is,
\begin{equation}\label{FRCAY-EQUATION4}
\sum_{(x,y)\in N}\cos\left(t(\lambda_x-\lambda_y)\right)=\frac{|\Gamma|^2-2|\Gamma|}{4}.
\end{equation}

Note that the number of  ordered pairs $(x,y)$ such that $\mathrm{wt}\left(\frac{2a(x-y)}{n}\right)$ is even is $|\Gamma|^2/2$. Then the number of $(x,y)$ such that $x\neq y$ and $\mathrm{wt}\left(\frac{2a(x-y)}{n}\right)$  is even  is $|\Gamma|^2/2-|\Gamma|$. Thus the number of $(x,y)$ such that $x\dot>y$ and $\mathrm{wt}\left(\frac{2a(x-y)}{n}\right)$ is even  is $\frac{|\Gamma|^2/2-|\Gamma|}{2}=\frac{|\Gamma|^2-2|\Gamma|}{4}$, that is, $|N|=\frac{|\Gamma|^2-2|\Gamma|}{4}$. Since $\cos\left(t(\lambda_x-\lambda_y)\right)\leq1$, (\ref{FRCAY-EQUATION4}) holds if and only if $\cos\left(t(\lambda_x-\lambda_y)\right)=1$ for  $(x,y)\in N$.  That is, $\frac{t}{2\pi}(\lambda_x-\lambda_y)\in \mathbb{Z}$ for $(x,y)\in N$. This completes the proof.
\qed\end{Tproof}

In Theorem \ref{FRCAY-mian result}, if $r=1$, then $G=\Cay(\mathbb{Z}_n,S)$ is a circulant graph.  At this time, the eigenvalues of circulant graphs are $\lambda_x=\sum_{s\in S}\omega_n^{xs}$, $x\in \mathbb{Z}_n$. Then we have the following result.
\begin{cor}\label{FRCAY-cor1}
 Let
 $N=\{ (x,y)|x,y\in \mathbb{Z}_n, x>y, 2|(x-y)\}.$
 A circulant graph $G=\Cay(\mathbb{Z}_n, S)$ has FR at time $t$ between vertices $u$ and $v$ if and only if
 \begin{itemize}
\item[\rm(a)] $n$ is even;
\item[\rm(b)] $u$ and $v$ are antipodal vertices, that is, $u-v=n/2$;
\item[\rm(c)] $\frac{t}{2\pi}\sum\limits_{s\in S}(\omega_n^{xs}-\omega_n^{ys})\in \mathbb{Z}, \text{~for~all~}  (x,y)\in N$, where $\omega_{n}=\exp(2\pi \mathrm{i}/n)$.
\end{itemize}
Moreover, if these conditions hold, and $\sum\limits_{s\in S}(\omega_n^{xs}-\omega_n^{ys})$ are integers for all $(x,y)\in N$, then there is a minimum time $t=\frac{2\pi}{M}$ at which FR occurs between $u$ and $v$, where
$$M=\gcd\left(\left\{\sum\limits_{s\in S}(\omega_n^{xs}-\omega_n^{ys})\right\}_{(x,y)\in N}\right).$$
\end{cor}

\begin{figure*}[h]
\begin{minipage}{0.5\textwidth}
\begin{center}
	\begin{tikzpicture}[scale=0.8,auto,swap]
    \tikzstyle{blackvertex}=[circle,draw=black,fill=black]
    \tikzstyle{bluevertex}=[circle,draw=blue,fill=blue]
    \tikzstyle{greenvertex}=[circle,draw=green,fill=green]
    \tikzstyle{yellowvertex}=[circle,draw=yellow,fill=yellow]
    \tikzstyle{brownvertex}=[circle,draw=brown,fill=brown]
    \tikzstyle{redvertex}=[circle,draw=red,fill=red]

   \node [blackvertex,scale=0.5] (a1) at (60:2) {};
    \node [blackvertex,scale=0.5] (a2) at (120:2) {};
    \node [blackvertex,scale=0.5] (a3) at (180:2) {};
    \node [blackvertex,scale=0.5] (a4) at (240:2) {};
    \node [blackvertex,scale=0.5] (a5) at (300:2) {};
    \node [blackvertex,scale=0.5] (a6) at (0:2) {};

    \node [scale=0.7] at (60:2.3) {0};
    \node [scale=0.7] at (120:2.3) {5};
    \node [scale=0.7] at (180:2.3) {4};
    \node [scale=0.7] at (240:2.3) {3};
    \node [scale=0.7] at (300:2.3) {2};
    \node [scale=0.7] at (0:2.3) {1};

    \draw [black,thick] (a1) -- (a2) -- (a3) -- (a4) -- (a5) -- (a6) -- (a1);
\end{tikzpicture}
\caption{$\Cay(\mathbb{Z}_6, S_1)$}
\label{figure1}
\end{center}
\end{minipage}%
\begin{minipage}{0.5\textwidth}
\begin{center}
	\begin{tikzpicture}[scale=0.8,auto,swap]
    \tikzstyle{blackvertex}=[circle,draw=black,fill=black]
    \tikzstyle{bluevertex}=[circle,draw=blue,fill=blue]
    \tikzstyle{greenvertex}=[circle,draw=green,fill=green]
    \tikzstyle{yellowvertex}=[circle,draw=yellow,fill=yellow]
    \tikzstyle{brownvertex}=[circle,draw=brown,fill=brown]
    \tikzstyle{redvertex}=[circle,draw=red,fill=red]

   \node [blackvertex,scale=0.5] (a1) at (90:2) {};
    \node [blackvertex,scale=0.5] (a2) at (180:2) {};
    \node [blackvertex,scale=0.5] (a3) at (270:2) {};
    \node [blackvertex,scale=0.5] (a4) at (0:2) {};

    \node [scale=0.7] at (90:2.3) {0};
    \node [scale=0.7] at (180:2.3) {1};
    \node [scale=0.7] at (270:2.3) {2};
    \node [scale=0.7] at (0:2.3) {3};

    \draw [black,thick] (a1) -- (a2) -- (a3) -- (a4)-- (a1);
\end{tikzpicture}
\caption{$\Cay(\mathbb{Z}_4, S_2)$}
\label{figure2}
\end{center}
\end{minipage}
\end{figure*}

Note \cite[Theorem 7.5]{CCTVZ19} that a cycle $\Cay(\mathbb{Z}_n, S)$, where $S=\{1,n-1\}$, has FR if and only if it has four or six vertices. In the following, we verify this result by Corollary \ref{FRCAY-cor1}.

\begin{example}
{\em Let  $S_1=\{1,5\}$. Then $G_1=\Cay(\mathbb{Z}_6, S_1)$ is a cycle of order six (as shown in Figure \ref{figure1}). It is easy to verify that
$$N=\{(2,0), (3,1),(4,0),(4,2),(5,1),(5,3)\}.$$
By a simple calculation, we have
$$\sum\limits_{s\in S_1}(\omega_n^{xs}-\omega_n^{ys})=2\cos\left(\frac{\pi}{3}x\right)-2\cos\left(\frac{\pi}{3}y\right).$$
For  $(x,y)\in N$, the corresponding values of $2\cos(\frac{\pi}{3}x)-2\cos(\frac{\pi}{3}y)$ are $-3, -3, -3, 0, 0, 3$, respectively. Therefore, $G_1$ has FR at time $t=\frac{2}{3}\pi$ between antipodal vertices by Corollary \ref{FRCAY-cor1}.  Moreover, if $t<\frac{2}{3}\pi$, then $G_1$ has no FR.
}
\end{example}

\begin{example}
{\em
 Let $S_2=\{1,3\}$. Then $G_2=\Cay(\mathbb{Z}_4, S_2)$ is a cycle of order four (as shown in Figure \ref{figure2}). It is easy to verify that
$$N=\{(2,0), (3,1)\}.$$
By a simple calculation, we have
$$\sum\limits_{s\in S_2}(\omega_n^{xs}-\omega_n^{ys})=2\cos\left(\frac{\pi}{2}x\right)-2\cos\left(\frac{\pi}{2}y\right).$$
For  $(x,y)\in N$, the corresponding values of $2\cos\left(\frac{\pi}{2}x\right)-2\cos\left(\frac{\pi}{2}y\right)$ are $-4, 0$, respectively. Therefore, $G_2$ has FR at time $t=\frac{\pi}{2}$ between antipodal vertices by Corollary \ref{FRCAY-cor1}.  Moreover, if $t<\frac{\pi}{2}$, then $G_2$ has no FR.
}
\end{example}

For a rational number $0\neq \mu=\frac{p}{q}\in \mathbb{Q}$, where $p$ and $q$ are coprime and $q>0$, we let $N(\mu)$ denote $q$.  The definition of $N(\mu)$ and the following result comes from \cite{Berger18}.
\begin{lemma}\label{FRCAY-LEMMA1}\emph{(see \cite[Theorem 1.2]{Berger18})}
Suppose $\mu_1$ and $\mu_2$ are rational numbers such that $\mu_1\pm\mu_2$ are not integers. Then the set $\{1,\cos(\mu_1\pi), \cos(\mu_2\pi)\}$ is linearly independent over $\mathbb{Q}$ if and only if $N(\mu_1),N(\mu_2)\geq4$ and $(N(\mu_1), N(\mu_2))\neq(5,5)$.
\end{lemma}

\begin{theorem}\emph{(see \cite[Theorem 7.5]{CCTVZ19})}
A cycle $\Cay(\mathbb{Z}_n, S)$, where $S=\{1,n-1\}$, does not have FR if $n\not=4,6$.
\end{theorem}

\begin{proof}
It is easy to verify that
$$N=\{(x,y)|x,y\in \mathbb{Z}_n, x>y, 2|(x-y)\}.$$
The condition (c) in Corollary \ref{FRCAY-cor1}  implies that
\begin{equation}\label{equation40}
\frac{\frac{t}{2\pi}\left(\sum\limits_{s\in S}(\omega_n^{xs}-\omega_n^{ys})\right)}{\frac{t}{2\pi}\left(\sum\limits_{s\in S}(\omega_n^{x's}-\omega_n^{y's})\right)}=
\frac{\cos\left(\frac{2\pi x}{n}\right)-\cos\left(\frac{2\pi y}{n}\right)}{\cos\left(\frac{2\pi x'}{n}\right)-2\cos\left(\frac{2\pi y'}{n}\right)}\in \mathbb{Q}, ~~\forall (x,y), (x',y')\in N.
\end{equation}
Hence it is sufficient to prove that (\ref{equation40}) does not hold for some $(x,y), (x',y')\in N$, in order to illustrate a cycle $C_n$ has no FR when $n\neq 4,6$.

\noindent\emph{Case 1.} $n\equiv 2 \pmod 4$. If $n\neq 6,10$, then $N(\frac{4}{n})=N(\frac{8}{n})\geq7$. By Lemma \ref{FRCAY-LEMMA1}, the set $\{1,\cos(\frac{4}{n}\pi), \cos(\frac{8}{n}\pi)\}$ is linearly independent over $\mathbb{Q}$. Let $(x,y)=(2,0)$ and $(x',y')=(4,0)$. Then
$$\frac{\cos(\frac{4}{n}\pi)-1}{\cos(\frac{8}{n}\pi)-1}\notin \mathbb{Q}.$$
If $n=10$, then
$$\frac{\cos(\frac{4}{n}\pi)-1}{\cos(\frac{8}{n}\pi)-\cos(\frac{4}{n}\pi)}=\frac{\frac{-1+\sqrt{5}}{4}-1}{\frac{-1-\sqrt{5}}{4}-\frac{-1+\sqrt{5}}{4}}=\frac{\sqrt{5}-1}{2}\notin \mathbb{Q}.$$

\noindent\emph{Case 2.} $n\equiv 0 \pmod 4$. For $n\geq 32$, by Lemma \ref{FRCAY-LEMMA1}, the set $\{1, \cos(\frac{4}{n}\pi), \cos(\frac{8}{n}\pi)\}$ is linearly independent over $\mathbb{Q}$. Then
$$\frac{\cos(\frac{4}{n}\pi)-1}{\cos(\frac{8}{n}\pi)-1}\notin \mathbb{Q}.$$
If $n=16, 20, 24, 28$, then
$$\frac{\cos(\frac{4}{n}\pi)-1}{\cos(\frac{2\pi\cdot n/2}{n})-1}=\frac{\cos(\frac{4}{n}\pi)-1}{-1-1}=\frac{1}{2}(1-\cos(\frac{4}{n}\pi))\notin \mathbb{Q}.$$
If $n=8,12$, then
$$\frac{\cos(\frac{6}{n}\pi)-\cos(\frac{2}{n}\pi)}{\cos(\frac{2\pi\cdot n/2}{n})-1}=\frac{\cos(\frac{6}{n}\pi)-\cos(\frac{2}{n}\pi)}{-1-1}=\frac{1}{2}(\cos(\frac{2}{n}\pi)-\cos(\frac{4}{n}\pi))\notin \mathbb{Q}.$$

Therefore, $\Cay(\mathbb{Z}_n, S)$, where $S=\{1,n-1\}$, does not have FR if $n\not=4,6$. \qed
\end{proof}

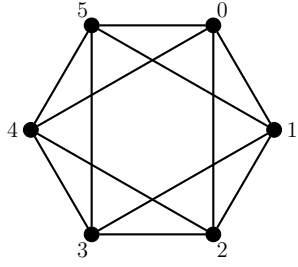
\begin{figure*}[h]
\begin{center}
	\begin{tikzpicture}[scale=0.8,auto,swap]
    \tikzstyle{blackvertex}=[circle,draw=black,fill=black]
    \tikzstyle{bluevertex}=[circle,draw=blue,fill=blue]
    \tikzstyle{greenvertex}=[circle,draw=green,fill=green]
    \tikzstyle{yellowvertex}=[circle,draw=yellow,fill=yellow]
    \tikzstyle{brownvertex}=[circle,draw=brown,fill=brown]
    \tikzstyle{redvertex}=[circle,draw=red,fill=red]

   \node [blackvertex,scale=0.5] (a1) at (60:2) {};
    \node [blackvertex,scale=0.5] (a2) at (120:2) {};
    \node [blackvertex,scale=0.5] (a3) at (180:2) {};
    \node [blackvertex,scale=0.5] (a4) at (240:2) {};
    \node [blackvertex,scale=0.5] (a5) at (300:2) {};
    \node [blackvertex,scale=0.5] (a6) at (0:2) {};

    \node [scale=0.7] at (60:2.3) {0};
    \node [scale=0.7] at (120:2.3) {5};
    \node [scale=0.7] at (180:2.3) {4};
    \node [scale=0.7] at (240:2.3) {3};
    \node [scale=0.7] at (300:2.3) {2};
    \node [scale=0.7] at (0:2.3) {1};

    \draw [black,thick] (a1) -- (a2) -- (a3) -- (a4) -- (a5) -- (a6) -- (a1);
    \draw [black,thick] (a1)-- (a3) -- (a5) -- (a1);
    \draw [black,thick] (a2)-- (a4) -- (a6)-- (a2) ;
\end{tikzpicture}
\caption{$\Cay(\mathbb{Z}_6, S_3)$}
\label{figure3}
\end{center}
\end{figure*}

\begin{example}
{\em Let  $S_3=\{1,2,4,5\}$ and $G_3=\Cay(\mathbb{Z}_6, S_3)$ (as shown in Figure \ref{figure3}). In this case,
$$
N=\{(2,0), (3,1),(4,0),(4,2),(5,1),(5,3)\}.
$$
By a simple calculation, we have
$$\sum\limits_{s\in S_3}(\omega_n^{xs}-\omega_n^{ys})=2\cos\left(\frac{\pi}{3}x\right)-2\cos\left(\frac{\pi}{3}y\right)+2\cos\left(\frac{2\pi}{3}x\right)-2\cos\left(\frac{2\pi}{3}y\right).$$
For  $(x,y)\in N$, the corresponding values of $\sum\limits_{s\in S_3}(\omega_n^{xs}-\omega_n^{ys})$ are $-6,0,-6,0,0,0$, respectively. Therefore, $G_3$ has FR at time $t=\frac{\pi}{3}$ between antipodal vertices by Corollary \ref{FRCAY-cor1}. Moreover, if $t<\frac{\pi}{3}$, then $G_3$ has no FR.
}
\end{example}

If $n_1=n_2=\cdots=n_r=2$ in Theorem \ref{FRCAY-mian result}, then $G=\Cay(\Gamma,S)$ is a cubelike graph. It is known \cite{CC} that the eigenvalues of cubelike graphs are $\lambda_x=\sum_{s\in S}(-1)^{s\cdot x}$, $x\in G$, where $s\cdot x=\sum_{i=1}^rs_ix_i$. Then we have the following result.
\begin{cor}\label{FRCAY-cor2}
Let $n=(2,2,\ldots,2)$. A cubelike graph $G=\Cay(\Gamma,S)$ has FR at time $t$ between $v$ and $v+a$ if and only if
\begin{equation*}
\frac{t}{2\pi}\sum_{s\in S}((-1)^{s\cdot x}-(-1)^{s\cdot y})\in \mathbb{Z}, \text{~for~all~} (x,y)\in N,
\end{equation*}
where $N=\left\{(x,y)\left| x,y\in G, x\dot>y, \mathrm{wt}\left(\frac{2a(x-y)}{n}\right) \text{~is~even}\right.\right\}$ and $a$ is of order $2$.
Moveover, if these conditions hold, then there is a  minimum time $t=\frac{2\pi}{M}$ at which FR occurs, where $M=\gcd\{\sum_{s\in S}((-1)^{s\cdot x}-(-1)^{s\cdot y}) | (x,y)\in N\}$.
\end{cor}

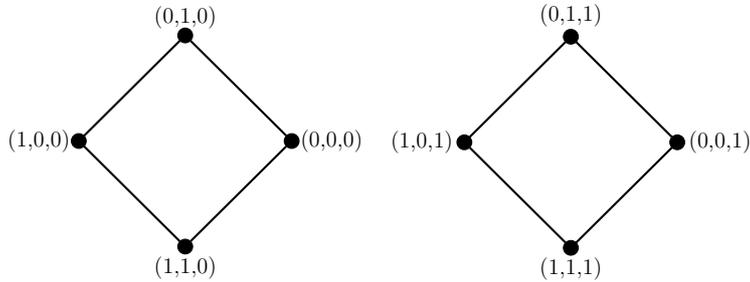
\begin{figure}[h]
\begin{center}
\begin{minipage}{0.1\textwidth}
	\begin{tikzpicture}[scale=0.7,auto,swap]
    \tikzstyle{blackvertex}=[circle,draw=black,fill=black]
    \tikzstyle{bluevertex}=[circle,draw=blue,fill=blue]
    \tikzstyle{greenvertex}=[circle,draw=green,fill=green]
    \tikzstyle{yellowvertex}=[circle,draw=yellow,fill=yellow]
    \tikzstyle{brownvertex}=[circle,draw=brown,fill=brown]
    \tikzstyle{redvertex}=[circle,draw=red,fill=red]

   \node [blackvertex,scale=0.5] (a1) at (0:2) {};
    \node [blackvertex,scale=0.5] (a2) at (90:2) {};
    \node [blackvertex,scale=0.5] (a3) at (180:2) {};
    \node [blackvertex,scale=0.5] (a4) at (270:2) {};

    \node [scale=0.7] at (0:2.75) {(0,0,0)};
    \node [scale=0.7] at (90:2.35) {(0,1,0)};
    \node [scale=0.7] at (180:2.75) {(1,0,0)};
    \node [scale=0.7] at (270:2.4) {(1,1,0)};

    \draw [black,thick] (a1) -- (a2) -- (a3) -- (a4) -- (a1) ;
\end{tikzpicture}
\end{minipage}%
\hspace{33mm}
\begin{minipage}{0.3\textwidth}
	\begin{tikzpicture}[scale=0.7,auto,swap]
    \tikzstyle{blackvertex}=[circle,draw=black,fill=black]
    \tikzstyle{bluevertex}=[circle,draw=blue,fill=blue]
    \tikzstyle{greenvertex}=[circle,draw=green,fill=green]
    \tikzstyle{yellowvertex}=[circle,draw=yellow,fill=yellow]
    \tikzstyle{brownvertex}=[circle,draw=brown,fill=brown]
    \tikzstyle{redvertex}=[circle,draw=red,fill=red]

   \node [blackvertex,scale=0.5] (a1) at (0:2) {};
    \node [blackvertex,scale=0.5] (a2) at (90:2) {};
    \node [blackvertex,scale=0.5] (a3) at (180:2) {};
    \node [blackvertex,scale=0.5] (a4) at (270:2) {};

    \node [scale=0.7] at (0:2.8) {(0,0,1)};
    \node [scale=0.7] at (90:2.4) {(0,1,1)};
    \node [scale=0.7] at (180:2.8) {(1,0,1)};
    \node [scale=0.7] at (270:2.4) {(1,1,1)};

    \draw [black,thick] (a1) -- (a2) -- (a3) -- (a4) -- (a1) ;
\end{tikzpicture}
\end{minipage}
\caption{$\Cay(\mathbb{Z}_2\oplus \mathbb{Z}_2\oplus\mathbb{Z}_2, S_4)$}
\label{figure4}
\end{center}
\end{figure}

\begin{example}
{\em Let $\Gamma=\mathbb{Z}_2\oplus \mathbb{Z}_2\oplus\mathbb{Z}_2$ and $S_4=\{(0,1,0), (1,1,0)\}$. Then $G_4=\Cay(\Gamma, S_4)$ is a disjoint union of two cycles of order $4$ (as shown in Figure \ref{figure4}). Let $a=(1,1,1)$.
It is easy to verify that
\begin{align*}
N=&\{((0,1,1),(0,0,0)),((1,1,0),(0,0,0)),((1,0,1),(0,0,0)),((0,1,0),(0,0,1)),\\
&((1,0,0),(0,1,0)),((1,1,1),(0,1,0)),((1,0,0),(0,0,1)),((1,1,1),(0,0,1)),\\
&((1,1,0),(0,1,1)),((1,0,1),(0,1,1)),((1,1,1),(1,0,0)),((1,1,0),(1,0,1))\}.
\end{align*}
Note that $x=(x_1,x_2,x_3)$ and $y=(y_1,y_2,y_3)$.  Then
$$\sum_{s\in S}((-1)^{s\cdot x}-(-1)^{s\cdot y})=(-1)^{x_2}-(-1)^{y_2}+(-1)^{x_1+x_2}-(-1)^{y_1+y_2}.$$
For  $(x,y)\in N$, the corresponding values of $\sum_{s\in S}((-1)^{s\cdot x}-(-1)^{s\cdot y})$ are $-4,-2,-2,-4,2,2$, $-2,-2,2,2,0,0$, respectively. Therefore, $G_4$ has FR at time $t=\pi$ between vertices $v$ and $v+(1,1,1)$ by Corollary \ref{FRCAY-cor2}. Moreover, if $t<\pi$, then $G_4$ has no FR.
}
\end{example}

Let $B_r$ be the ring of Boolean functions with $r$ variables. Let $r=2m+1~(m\geq2)$, and $f$ be a bent function (see \cite[Defintion 4.4]{Tan19}) in $B_{r-1}$. Denote $$
S_1=\{z_1\in \mathbb{Z}_2^{r-1}:f(z_1)=1\} \text{~with~}0\notin S_1,~ S^{\varepsilon}=(\varepsilon, S_1)~(\varepsilon=0,1) \text{~and~}S=S^0\cup S^1.
$$
The Cayley graph $G=\Cay(\mathbb{Z}_2^r,S)$ was defined in \cite{Tan19}, where the perfect state transfer of $G=\Cay(\mathbb{Z}_2^r,S)$ was considered. Note that $G=\Cay(\mathbb{Z}_2^r,S)$ is a cubelike graph. Next we prove $G=\Cay(\mathbb{Z}_2^r,S)$ has FR by Theorem \ref{FRCAY-mian result}.
\begin{cor}
Let $G=\Cay(\mathbb{Z}_2^r,S)$ be defined as above. Then $G$ has FR at time $t=\frac{\pi}{2^m}$ between vertices $v$ and $v+(1,0)$, $0\in \mathbb{Z}_2^{r-1}$.
\end{cor}
\begin{proof}
It is easy to verify that
$$N=\{(x,y)|\left((0,x_1), (0,y_1)\right), \left((1,x_1), (1,y_1)\right), x_1\dot>y_1, x_1,y_1\in \mathbb{Z}_2^{r-1}\}.$$
The eigenvalues of $G$ are given as follows  \cite{Tan19}:
\begin{align*}
\lambda_x=\left\{
\begin{array}{ccc}
2^{r-1}-W_f(x_1), &\text{~~~~~if~} x=(0,x_1) \text{~and~}x_1=0\in \mathbb{Z}_2^{r-1},\\[0.2cm]
-W_f(x_1),        &\text{~~~~~if~} x=(0,x_1) \text{~and~}x_1\neq0\in \mathbb{Z}_2^{r-1},\\[0.2cm]
0,             &\text{if~} x=(1,x_1) \text{~and~}x_1\in \mathbb{Z}_2^{r-1}.
\end{array}
\right.
\end{align*}
where $W_f(x_1)=2^m$ or $-2^m$. Then
$$\lambda_x-\lambda_y\in \{0,-2^{r-1}, \pm2^{m+1}\}, ~ \forall(x,y)\in N.$$
Since $r-1=2m>m$, $M=\gcd\{(\lambda_x-\lambda_y)_{(x,y)\in N}\}=2^{m+1}$. By Theorem \ref{FRCAY-mian result}, $G$ has FR  at time $t=\frac{\pi}{2^m}$ between vertices $v$ and $v+(1,0)$, $0\in \mathbb{Z}_2^{r-1}$.
\end{proof}

\section*{Acknowledgements}
The authors greatly appreciate the anonymous referees for their comments and suggestions.

\section*{Statements and Declarations}

\noindent \textbf{Competing interests}~~No potential competing interest was reported by the authors.

\medskip

\noindent \textbf{Data availability statements}~~Data sharing not applicable to this article as no datasets were generated or analysed during the current study.

\end{document}